\documentclass[11pt]{article}
\usepackage[left=3cm,right=3cm,top=3cm,bottom=3cm]{geometry}
\usepackage{amsmath,amssymb,amsthm}
\numberwithin{equation}{section}
\newtheorem{thm}{Theorem}[section]
\newtheorem{theorem}[thm]{Theorem}

\newtheorem{lemma}[thm]{Lemma}

\newtheorem{proposition}[thm]{Proposition}

\theoremstyle{definition}

\theoremstyle{remark}

\newtheorem{remark}[thm]{Remark}

\newtheorem{example}[thm]{Example}

\newcommand{\N}{\mathbb{N}}
\newcommand{\Z}{\mathbb{Z}}

\newcommand{\R}{\mathbb{R}}
\newcommand{\C}{\mathbb{C}}

\newcommand{\la}{\langle}
\newcommand{\ra}{\rangle}

\newcommand{\supp}{\mathrm{supp}}

\newcommand{\sfrac}[2]{#1/#2}
\newcommand{\ddp}[2]{\frac{\partial #1}{\partial #2}}

\newcommand{\half}{{\frac{1}{2}}}
\newcommand{\thalf}{{\tfrac{1}{2}}}

\newcommand{\tquarter}{{\tfrac{1}{4}}}
\newcommand{\spec}{\mathrm{spec}}
\newcommand{\msq}{\smash{m^2}}
\begin{document}
\title{A simple method for finite range decomposition \\ of quadratic forms and Gaussian fields}
\author{Roland Bauerschmidt}
%
\date{March 9, 2019}
\maketitle
\hyphenation{qua-dra-tic semi-def-in-ite-ness}
\begin{abstract}
  We present a simple method to decompose the Green forms corresponding to a large
  class of interesting symmetric Dirichlet forms into integrals over
  symmetric positive semi-definite and finite range (properly
  supported) forms that are smoother than the original Green form.
  This result gives rise to multiscale decompositions
  of the associated Gaussian free fields into sums of independent
  smoother Gaussian fields with spatially localized correlations.
  Our method makes use of the finite
  propagation speed of the wave equation and Chebyshev polynomials.
  It improves several existing results and also gives simpler
  proofs.
\end{abstract}

%

\section{Introduction and main result}

\subsection{The Newtonian potential}

Let us place the result of this paper into context through an example.
Consider the Newtonian potential, the Green's function of the Laplace
operator on $\R^d$ given by
\begin{equation} \label{eq:Newton-defn}
  \Phi(x) = C_d \begin{cases}
    |x|^{-(d-2)} & (d \geq 3)\\
    \log 1/|x| & (d=2)
  \end{cases}
  \quad \text{for all $x \in \R^d$, $x \neq 0$.}
\end{equation}

For $d\geq 3$ and \emph{any} measurable function $\varphi : \R \to \R$
such that $t^{d-3}\varphi(t)$ is integrable, the Newtonian potential can be
written, up to a constant, as
\begin{equation} \label{eq:Newton-decomp}
  |x|^{-(d-2)} = \int_0^\infty t^{-(d-2)} \, \varphi(|x|/t) \; \frac{dt}{t}
  \quad \text{for all $x \in \R^d$, $x\neq 0$.}
\end{equation}
This is true because both sides are radially symmetric and homogeneous
of degree $-(d-2)$, where homogeneity of the right-hand side simply
follows from the change of variables formula.
In particular,
$\varphi$ can be chosen smooth with compact support and such that
$\varphi(|x|)$ is a positive semi-definite function on $\R^d$.  The
last condition means that $\varphi(|x|)$ is positive as a quadratic form:
for any $f \in C_c^\infty(\R^d)$, that is, $f: \R^d \to \R$ smooth
with compact support,
\begin{equation} \label{eq:S-varphi}
  \Phi_t(f,f) := \int_{\R^d \times \R^d} \varphi(|x-y|/t) f(x)f(y) \; dx \, dy \geq 0.
\end{equation}

Similarly, if $d=2$, and $\varphi: \R \to \R$ is any absolutely
continuous function with $\varphi(0) =1$ and such that $\varphi'(t)$ is integrable,
then
\begin{equation} \label{eq:Newton-decomp-2}
  \log 1/|x| = \int_0^\infty (\varphi(|x|/t) - \varphi(1/t)) \; \frac{dt}{t}
  \quad \text{for all $x \in \R^2$, $x \neq 0$}
  .
\end{equation}
Indeed, for $x\neq 0$,
\begin{equation}
  \log 1/|x|
  = \varphi(0) \log 1/|x|
  = -\int_0^\infty  \varphi'(s) \log 1/|x|  \; ds
  = \int_0^\infty  \varphi'(s) \int_{s/|x|}^s \frac{dt}{t} \; ds
  ,
\end{equation}
and thus, since $\varphi'$ is integrable, by Fubini's theorem,
\begin{equation}
  \log 1/|x|
  =
  \int_0^\infty \int_{t}^{t|x|} \varphi'(s) \; ds \; \frac{dt}{t}
  =
  \int_0^\infty (\varphi(t|x|)-\varphi(t)) \; \frac{dt}{t}
  ,
\end{equation}
showing \eqref{eq:Newton-decomp-2} after the change of variables $t \mapsto 1/t$.
Now suppose again that $\varphi$ is chosen such that $\varphi(|x|)$ is
a positive semi-definite function on $\R^2$. Then the function $\R^2
\ni x \mapsto \varphi(|x|/t)-\varphi(1/t)$ is positive as a quadratic
form on the domain of smooth and compactly supported functions with
vanishing integral:
\begin{align} \label{eq:Newton-decomp-2-pd}
  \Phi_{t}(f,f)
  &:=
  \int_{\R^2\times \R^2} (\varphi(|x-y|/t)-\varphi(1/t)) f(x)f(y) \; dx \, dy 
  \\
  &= \int_{\R^2 \times \R^2} \varphi(|x-y|/t) \; f(x)f(y) \; dx \, dy \geq 0
  \nonumber
\end{align}
for all $f \in C_c^\infty(\R^2)$ with $\int f \; dx = 0$.

The above shows that the Newtonian potentials \eqref{eq:Newton-defn}
admit decompositions into integrals of compactly supported and
positive semi-definite functions, with the appropriate restriction of
the domain for $d=2$.

Let us only remark at this point that the positivity of a quadratic
form has the important implication that it entails the existence of a
corresponding Gaussian process, discussed briefly in
Section~\ref{sec:gaussian-fields}. It is however also of interest in
mathematical physics for different reasons \cite{MR1930084}.

\subsection{Finite range decompositions of quadratic forms}

It is an open problem to characterize the class of positive quadratic
forms, $S: D(S) \times D(S) \to \R$, that admit decompositions into
integrals (or sums)
of positive quadratic forms of finite range:
for all $f,g \in D(S)$, $t>0$,
\begin{align} \label{eq:S-decomp}
  &\left\{
  \begin{gathered}
  S(f,g) = \int_0^\infty S_t(f,g) \; \frac{dt}{t},
  \medskip
  \\
  S_t: D(S) \times D(S) \to \R,
  \\
  S_t(f,f) \geq 0,
  \\
  S_t(f,g) = 0
  \; \text{ if $d(\supp(f),\supp(g)) > \theta(t)$,}
\end{gathered}
\right.
\end{align}
where $\theta: (0,\infty) \to (0,\infty)$ is increasing and $d$ is a distance function.
The condition of \emph{finite range}, the last condition in
\eqref{eq:S-decomp},
generalizes the property of compact support of
the function $\varphi$ in \eqref{eq:S-varphi} to quadratic forms that
are not defined by a convolution kernel.
The difficulty in decomposing quadratic forms in such a way is to
achieve the two conditions of positivity and finite range
simultaneously.
Note that by splitting up the integral, one can obtain a decomposition into
a sum from \eqref{eq:S-decomp}, and conversely, a decomposition into a sum
can be written as an integral (without regularity in $t$).

For applications, not only the existence, but also the regularity of
the decomposition \eqref{eq:S-decomp} is important.
Let $(X,\mu)$ be a metric measure space, i.e.,
a locally compact complete separable metric space $X$ with a Radon measure $\mu$ on $X$
with full support (i.e., $\mu$ is strictly positive),
$C_c(X)$ the space of continuous functions on $X$ with compact support,
and $C_b(X)$ the space of bounded and continuous functions on $X$.
Let us say that the decomposition \eqref{eq:S-decomp} is regular if
$C_c(X) \cap D(S)$ is $S$-dense in $D(S)$ and if every $S_t$ has a
bounded continuous kernel $s_t \in C_b(X \times X)$:
\begin{equation}
  S_t(f,g) = \int s_t(x,y) f(x)g(y) \; d\mu(x) \, d\mu(y)
  \quad \text{for all $f,g \in C_c(X) \cap D(S)$}.
\end{equation}

For the decompositions \eqref{eq:Newton-decomp},
\eqref{eq:Newton-decomp-2}, the kernels are of course given in terms
of the smooth function $\varphi$ by the explicit formula
\begin{equation} \label{eq:s-Newton}
  \phi_t(x,y) = t^{-(d-2)} \varphi(|x-y|/t)
  \quad \text{for all $x,y \in \R^d$, $t>0$}.
\end{equation}
Note that for $d=2$ the second term in \eqref{eq:Newton-decomp-2} could be omitted
by \eqref{eq:Newton-decomp-2-pd}, with the understanding that the
quadratic form is restricted to functions with vanishing integral.
It follows in particular that
\begin{equation} \label{eq:s-decay}
  |\phi_t(x,y)| \leq C t^{-(d-2)}
  \quad \text{uniformly in all $x,y \in \R^d$.}
\end{equation}
This reflects the decay of the Newtonian potential. Moreover, for all integers $l_x, l_y \geq 0$,
the derivatives of the kernel $s_t$ decay according to
\begin{equation} \label{eq:grad-s-decay}
  |D_x^{l_x} D_y^{\smash{l_y}} \phi_t(x,y)| \leq C_l t^{-(d-2)} t^{-l_x-l_y},
\end{equation}
reflecting that $|D^l\Phi(x)| \leq C_l |x|^{-(d-2-l)}$ for all $x \in \R^d$, $x\neq 0$.

\medskip

The main result of this paper is a rather simple construction of
decompositions \eqref{eq:S-decomp} with estimates like
\eqref{eq:s-decay} for quadratic forms that arise by duality with
Dirichlet forms in a large class. Let us call such forms
\emph{Green forms} motivated by the Newtonian potential, or Green's
function, that is a special case.
This is explained in Section~\ref{sec:symmetric-forms}.

The main idea of our method is that \eqref{eq:S-decomp} can be
achieved by applying formulae like \eqref{eq:Newton-decomp} to the
spectral representation of the Green form, and then exploiting
finite propagation speed properties of appropriate wave flows.  These
are generalizations of the fact that if $u(t,x)$ is a solution to
\begin{equation} \label{eq:wave-equation}
  \partial_t^2 u-\Delta u = 0, \quad u(0,x) = u_0(x), \; \partial_t u(0,x) = 0
\end{equation}
with compactly supported initial data $u_0$ that then
\begin{equation} \label{eq:wave-prop-speed}
  \supp(u(t, \cdot)) \subseteq N_t(\supp(u_0))
\end{equation}
where $N_t(U) = \{x \in X: d(x,U) \leq t \}$ for any $U \subset X$.

The idea of exploiting properties of the wave equation in the context
of probability theory is not new.  For example, Varopoulos
\cite{MR822826} has
used the finite propagation speed of the wave equation to obtain Gaussian
bounds on the heat kernel of Markov chains, by decomposing it into
compactly supported pieces. 
Our objective is slightly different in that we are interested in
the constraint of positive definite decompositions.

Decompositions of singular functions into sums or integrals of smooth
and compactly supported functions have a history in analysis.  For
example, Fefferman's celebrated proof of pointwise almost everywhere
convergence of the Fourier series \cite{MR0340926} uses a
decomposition of $1/x$ on $\R$ like \eqref{eq:Newton-decomp}, albeit
without using positive semi-def\-in\-ite\-ness.
Hainzl and Seiringer \cite{MR1930084}, motivated by applications to
quantum mechanics such as \cite{MR864658}, decompose
general radially symmetric functions, without assuming a priori that they are
positive definite, into weighted integrals over \emph{tent functions}.
These, like $\varphi(|x|)$ in \eqref{eq:Newton-decomp}, are positive
semi-definite.  They state sufficient conditions for
the weight to be non-negative, and thus obtain decompositions like
\eqref{eq:Newton-decomp} for a class of radially
symmetric potentials including $e^{-m|x|}/|x|$ on $\R^3$.
Special cases and similar results have also
appeared in earlier works of P{\'o}lya \cite{MR0028541} and of
Gneiting \cite{MR1701408,MR1823914}.

These results, like \eqref{eq:Newton-decomp}, make essential use of
radial symmetry. One example of particular interest for probability
theory---where radial symmetry is not given---is the Green's function of
the discrete Laplace operator:
\begin{equation} \label{eq:discrete-Laplace}
  \Delta_{\Z^d}u(x) = 
  \sum_{e \in \Z^d: |e|_1=1} (u(x+e)-u(x))
  \quad \text{for any $u: \Z^d \to \R$, $x \in \Z^d$,}
\end{equation}
Brydges, Guadagni, and Mitter \cite{MR2070102} showed
that also in this discrete case, the corresponding Green's function, or more generally the
resolvent, admits a decomposition like \eqref{eq:S-decomp}
into a sum (instead of an integral) of positive semi-definite lattice functions
with estimates analogous to \eqref{eq:grad-s-decay}.
Brydges and Talaczyck \cite{MR2240180} gave a related construction
which applies to quite general elliptic operators on domains in
$\R^d$, but estimates on the kernels of this decomposition are only
known when the coefficients are constant.
Their construction was adapted by Adams, Koteck\'y, and M\"uller
\cite{AKM12a} to show that the Green's functions of constant coefficient
discrete elliptic systems on $\Z^d$ admit decompositions with estimates
analogous to \eqref{eq:grad-s-decay} and that the decomposition obtained this way
is analytic as a function of the (constant) coefficients.  These results are all based on a
constructions that average Poisson kernels.

Our method, as briefly sketched earlier, is
different from that of
\cite{MR2070102,MR2240180,BrydgesMitter12,AKM12a} and
yields simpler proofs of their results about constant coefficient
elliptic operators---both in discrete and continuous context.
It furthermore naturally yields a decomposition into an integral instead of a sum
(with integrand smooth in $t$),
and gives effective estimates for decompositions of Green's functions
of variable coefficient operators.

\subsection{Duality and spectral representation of the Green form}
\label{sec:symmetric-forms}

Let us now introduce the general set-up in which our result is framed
more precisely.  For motivation, we first return to the quadratic
forms defined by the Newtonian potentials \eqref{eq:Newton-defn}:
\begin{equation} \label{eq:Newton-form}
  \Phi(f,g)
  := \int_{\R^d \times \R^d} \Phi(x-y)f(x)g(y) \; dx \, dy
  ,
  \quad f,g \in D(\Phi)
\end{equation}
where
\begin{equation} \label{eq:Newton-form-domain}
  \begin{cases}
    D(\Phi) =  C_c^\infty(\R^d) & (d\geq 3)\\
    D(\Phi) =  \{ f \in C_c^\infty(\R^2): \int_{\R^2} f \; dx = 0\} & (d=2).
  \end{cases}
\end{equation}
These quadratic forms are not bounded on $L^2(\R^d)$, as is most
apparent when $d=2$.  They are closely related to the Dirichlet forms
given by
\begin{equation} \label{eq:Newton-Dirichlet}
  E(u,v)
  := \int_{\R^d} \nabla u \cdot \nabla v \; dx
  ,
  \quad
  u, v\in C_c^\infty(\R^d).
\end{equation}
The correspondence between the two is \emph{duality}: for all $f \in D(\Phi)$,
\begin{equation} \label{eq:Newton-duality}
  \Phi(f,f) = \sup \left\{ \int_{\R^d} fu \; dx: u \in C_c^\infty(\R^d), E(u,u) \leq 1 \right\}^{2}.
\end{equation}

This set-up admits the following natural generalization: Let $(X,\mu)$
always be a metric measure space and $L^2(X)$ be the Hilbert
space of equivalence classes of real-valued square $\mu$-integrable
functions on $X$ with inner product $(u,v)=(u,v)_{L^2}$.
Let $E: D(E) \times D(E) \to \R$ be a closed positive quadratic form
on $L^2(X)$ with $D(E) \subseteq L^2(X)$ a dense linear subspace.
It is sometimes convenient to assume that $E$ is regular, i.e., that
$C_c(X) \cap D(E)$ is $E$-dense in $D(E)$.
That $E$ is closed means that $D(E)$ is a Hilbert space
with inner product $E(u,v)+m^2(u,v)_{L^2}$ for any $m^2 >0$.
For the example \eqref{eq:Newton-Dirichlet}, the domain of the form closure
$D(E)$ of $C_c^\infty(\R^d)$ is the usual Sobolev space $H^1(\R^d)$ and
$(u,v)_{H^1} = E(u,v)+(u,v)_{L^2}$ is the usual Sobolev inner product.

It follows \cite{MR751959} from closedness that $E$ is the quadratic form
associated to a unique self-adjoint operator $L: D(L) \to L^2(X)$,
\begin{equation}
  E(u,v) = (u,Lv) \quad \text{for $u \in D(E)$, $v \in D(L)$}
  ,
\end{equation}
where $D(L)\subseteq D(E)$ is a dense linear subspace in $L^2(X)$.
Moreover, self-adjointness of $L$ gives rise to a spectral family and functional calculus.
This means in particular that for any Borel measurable $F : [0,\infty) \to \R$, there is a self-adjoint operator,
denoted $F(L) : D(F(L)) \to L^2(X)$, where
\begin{align}
  F(L) &:= \int_0^\infty F(\lambda) \; dP_\lambda,
  \\
  D(F(L)) &:= \left\{ u \in L^2(X): \int_0^\infty F(\lambda)^2 \; d(u, P_\lambda u)  < \infty \right\}
\end{align}
with $P_\lambda$ the spectral family associated to $L$, and $(u,
P_\lambda u)$ is the spectral measure associated to $L$ and $u \in
L^2(X)$.  In these terms, $E$ has the representation
\begin{equation} \label{eq:E-L12}
  E(u,u) = \|L^{\half}u\|_{L^2(X)}^{2}
  = \int_{\spec(L)} \lambda \; d(u,P_\lambda u),
  \quad u \in D(E) = D(L^{\half}),
\end{equation}
where $E(u,v)$ for $u \neq v$ is defined by the polarization identity.
Similarly, the corresponding Green form can be defined by
polarization and
\begin{equation} \label{eq:G-L12}
  \Phi(f,f) = \|L^{-\half}f\|_{L^2(X)}^{2}
  = \int_{\spec(L)} \lambda^{-1} \; d(u,P_\lambda u),
  \quad f\in D(\Phi) = D(L^{-\half}).
\end{equation}
This representation will be our starting point  
for the decomposition of the Green form.
Before stating the result and its proof, let us sketch how the
decomposition problem arises in probability theory.

\subsection{Gaussian fields and statistical mechanics}
\label{sec:gaussian-fields}

Even though the linear space $D(E)$ is complete under the metric induced
by the inner product $E(u,v)+m^2(u,v)_{L^2}$ for any $\msq > 0$, it is generally
not complete for $\msq = 0$. It may however be completed to a Hilbert
space abstractly; we denote this Hilbert space by $(H_E, (\cdot,\cdot)_E)$.
Similarly, we can complete the domain $D(\Phi)$ to a
Hilbert space under the quadratic form $\Phi$; this Hilbert space is
denoted by $(H_\Phi, (\cdot,\cdot)_\Phi)$.
$H_E$ and $H_\Phi$ are dual in the following sense:
The $L^2$ inner product can be restricted to 
\begin{equation}
  \la \cdot, \cdot \ra : D(\Phi) \times D(E) \to \R,
  \quad \la f, u \ra = (f,u) = (L^{-\half}f, L^{\half} u)
\end{equation}
which extends to a bounded bilinear form on $H_\Phi \times H_E$.
$L$ acts by definition isometric from $D(E)$ to $D(\Phi)$,
with respect to the norms of $H_E$ and $H_\Phi$, and
it extends to an isometric isometry from $H_E$
to $H_\Phi$. Thus $H_\Phi$ is identified with the dual space of $H_E$ naturally, via the
extension of the $L^2$ pairing $\la \cdot, \cdot \ra$.

\begin{remark}
To give some insight into the interpretation of the spaces $H_E$ and
$H_\Phi$, let us mention how $H_E$ can be characterized in the case of
the Newtonian potential \cite{MR2849840}:
\begin{multline}
  H_E \cong \{ f : \R^d \to \R \text{ measurable}:
  \\
  \text{there exists an $E$-Cauchy sequence $f_n \in D(E)$ with $f_n \to f$ a.e.} \} / \sim_d
\end{multline}
where $\sim_d$ is the usual identification of functions that are equal
almost everywhere when $d\geq 3$.
For $d=2$, $\sim_d$ in contrast identifies functions that may differ
by a constant almost everywhere.
(It is therefore sometimes said that the massless free field
does not exist in two dimensions, but that its gradient does.
The massless free field is the free field corresponding to $\Phi$ in the terminology explained below.)
To understand this distinction,
take a smooth cut-off function $\varphi_1$ on $\R^2$, e.g.\ with $\varphi_1
\equiv 1$ on $B_1(0)$ and $\varphi_1 \equiv 0$ on $B_2(0)^c$, set
$\varphi_n(x)=\varphi_1(x/n)$, and note that $E(\varphi_n,\varphi_n) =
n^{d-2} E(\varphi_1,\varphi_1)$. Thus, $(\varphi_n)$ is bounded in $H_E$ whenever
$d\leq 2$, and then (by the Banach-Alaoglu theorem) there is $\psi
\in H_E$ such that $\varphi_n \to \psi$ weakly along a subsequence in $H_E$; however,
$\varphi_n \to 1$ pointwise, so that $\psi \equiv 1 \in H_E$.
Now $E(1,1) = 0$ implies that the constant functions must be in the same equivalence
class as the zero function.
\end{remark}

It is well-known that any separable real Hilbert space $(H,(\cdot,
\cdot)_H)$ defines a Gaussian process indexed by $H$
\cite{MR544188}.  This is a probability space $(\Omega, P)$
and a unitary map $f \in H \to \la f, \phi \ra \in L^2(P)$ such
that the random variables $\la f, \phi \ra$ are Gaussian with
variance $(f,f)_H$.  Note that $\la f, \phi \ra$ is merely a
symbolic notation for the random variable on $L^2(P)$ that corresponds
to $f \in H$.  It cannot in general be interpreted as the pairing of
$f \in H$ with a random element $\phi(\omega) \in H$ defined for
$\omega \in \Omega$; see e.g.\ \cite{MR2322706}.

In particular, if $(H,(\cdot,\cdot)_H)$ is the Hilbert space
$(H_\Phi, (\cdot, \cdot)_{H_\Phi})$, this process
is called the \emph{free field} or the \emph{Gaussian free field}
(corresponding to Dirichlet form $E$ or Green's function $\Phi$).
The importance of free fields in statistical mechanics, and
probability theory in a wider sense, is well-recognized.  For
instance, observables of many models of statistical mechanics are
intricately related to them, by relations such as the the Kac--Siegert
transform \cite{MR2523458}. These models include spin models such as
the Ising model, as well as Coulomb and dipole systems.  In a
different direction, if $E$ is a Markovian form that satisfies
some regularity conditions, there exists an
associated Markov process \cite{MR2778606}, and it turns out that
there are strong connections between the distributions of the local
times of this Markov process and the free field associated to the same
Dirichlet form; see e.g.\
\cite{Syma69,MR648362,MR756768,MR734803,MR693227}.  In particular,
in a generalized ``non-commutative'' notion of Gaussian processes that
are supersymmetric, this correspondence becomes especially
striking; see e.g.\ the review \cite{MR2525670}.  The last mentioned
correspondence is the point of departure for an analysis of the
critical behavior of models of self-avoiding walks in dimension four
\cite{BS-saw4}.

For typical applications to statistical mechanics, the measure space
$(X,\mu)$ of Section~\ref{sec:symmetric-forms} is endowed with
additional structure such as a distance function,
a notion of smoothness, etc.\ as is the case for the
Newtonian potential.  The global properties of the free field are of
special interest for statistical mechanics.  An example of such a
global property is, if $X$ is an infinite graph, and $X_n \uparrow X$
is an increasing sequence of finite graphs approximating $X$, in an
appropriate sense, the behavior of
\begin{equation}
  \int \prod_{x\in X_n} e^{-V(\phi_x)} \; dP(\phi), \quad \text{as $n\to \infty$}
\end{equation}
for some $V: \R \to \R$.  The covariance $\Phi$ is typically 
long-range as in \eqref{eq:Newton-defn}.
This makes the analysis of the global properties of free fields difficult.

Decompositions like \eqref{eq:Newton-form} give rise to notions of
\emph{scale} and corresponding \emph{multiscale decompositions} of the
Gaussian free field and therefore provide a point of departure for
multiscale analysis. One instance of such an application is the
renormalization group method; see e.g.\ \cite{MR2523458} and references
therein.

\subsection{Main result}

Let $(X,\mu)$ be a metric measure space.
In addition, let $d : X \times X \to [0,\infty]$ be an extended pseudometric on $X$.
(\emph{Extended} means that $d(x,y)$ may be infinite
and \emph{pseudo} that $d(x,y) = 0$ for $x\neq y$ is allowed.
Example~\ref{ex:systems} below gives an example of interest where $d$ is not
the metric of $X$.)

Let $E: D(E) \times D(E) \to \R$ be a regular closed symmetric form on $L^2(X)$
as in Section~\ref{sec:symmetric-forms} and
denote by $L : D(L) \to L^2(X)$ the self-adjoint generator of $E$. 
Theorem~\ref{thm:decomposition} assumes that $(X,\mu,d,E)$
satisfies one of the following two \emph{finite propagation speed} conditions
that we now introduce:
For $\gamma > 0$, $B>0$, and an increasing function $\theta:(0,\infty)\to(0,\infty)$,
let us say that $(X,\mu,d,E)$ satisfies \eqref{eq:prop-speed-gamma} 
respectively \eqref{eq:polynom-gamma} if:
\begin{alignat}{1} \label{eq:prop-speed-gamma} \tag{$P_{\gamma,\theta}$}
  &\begin{aligned}\supp(\cos(L^{\half \gamma} t)u) \subseteq N_{\theta(t)}(\supp(u))
    \quad \text{for all $u \in C_c(X)$, $t > 0$,}
    \end{aligned}
  \\
  \intertext{respectively}
  \label{eq:polynom-gamma} \tag{$P_{\theta,B}^{\smash{*}}$}
  &\begin{aligned}
  &E(u,u) \leq B\|u\|_{L^2(X)}
  \quad \text{for all $u \in L^2(X)$},
  \\
  &\supp(L^{n}u) \subseteq N_{\theta(n)}(\supp(u))
  \quad \text{for all $u \in C_c(X)$, $n \in \N$,}
  \end{aligned}
\end{alignat}
where as before $N_t(U) = \{x \in X: d(x,U) \leq t \}$ for any $U \subset X$.
The left-hand side of \smash{\eqref{eq:prop-speed-gamma}} is defined
in terms of functional calculus for the self-adjoint operator $L$.

Note that if $L=-\Delta_{\R^d}$
is the Laplace operator of $\R^d$, then
$u(t,x) = [\cos(L^{\half}t)u_0](x)$ is a solution to the standard wave equation \eqref{eq:wave-equation},
and the condition \eqref{eq:prop-speed-gamma} with $\gamma=1$ and $\theta(t)=t$ is
the finite propagation speed property \eqref{eq:wave-prop-speed}.
The property
holds for more general elliptic operators and
elliptic systems (not necessarily of second order), however;
see Example~\ref{ex:systems} below.
Similarly, if $L =-\Delta_{\Z^d}$ is the discrete Laplace operator \eqref{eq:discrete-Laplace}, then
\smash{\eqref{eq:polynom-gamma}} holds with $B=2$ and $\theta(n) = n$, since
$Lu(x)$ only depends on $u(y)$ when $x$ and $y$ are nearest neighbors.
As for the property \eqref{eq:prop-speed-gamma}, the condition \eqref{eq:polynom-gamma}
remains true for more general discrete Dirichlet forms;
see Examples~\ref{ex:systems}--\ref{ex:graphs}.

\medskip

Let us introduce a further condition: The heat kernel bound
\eqref{eq:H-alpha} holds when the heat semigroup $(e^{-tL})_{t>0}$ has
continuous kernels $p_t$ for all $t>0$ and
there is $\alpha>0$ and a bounded function $\omega: X \to \R_+$ such that
\begin{equation} \label{eq:H-alpha} \tag{$H_{\alpha,\omega}$}
  p_t(x,x) \leq \omega(x) t^{-\alpha/2}
  \quad \text{for all $x \in X$.}
\end{equation}
Criteria for \eqref{eq:H-alpha} are classic; see e.g.\ \cite{MR0100158}
for second-order elliptic operators
and also the discussion in the examples below.

\begin{theorem} \label{thm:decomposition}
  Suppose $(X,\mu,d,E)$ satisfies \eqref{eq:prop-speed-gamma} or
  \eqref{eq:polynom-gamma}.  Then the corresponding Green form
  \eqref{eq:G-L12} admits a finite range decomposition
  \eqref{eq:S-decomp} with $S=\Phi$ and $S_t=\Phi_t$ 
  such that the $\Phi_t$ are bounded quadratic forms with
  \begin{equation}
    |\Phi_t(f,g)| \leq C_{\gamma,B} t^{\sfrac{2}{\gamma}} \|f\|_{L^2(X)}\|g\|_{L^2(X)}
    \quad \text{for all $f,g \in L^2(X)$.}
  \end{equation}
  Moreover, \eqref{eq:H-alpha} implies that the $\Phi_t$ have continuous kernels $\phi_t$ that satisfy
  \begin{equation} \label{eq:w-decay}
    |\phi_t(x,y)| \leq C_{\alpha,\gamma,B} \sqrt{\omega(x)\omega(y)} t^{-(\alpha-2)/\gamma}
    .
  \end{equation}
  In the discrete case \eqref{eq:polynom-gamma}, we used the convention $\gamma=1$
  and assumed that $t \geq 1$. For $t<1$ instead the explicit formula
  $\phi_t(x,y) = C {\bf 1}_{x=y} t$ holds.
\end{theorem}

\subsection{Examples}

\begin{example}[Elliptic operators with constant coefficients] \label{ex:constant-coefficients}
  Let $a = (a_{ij})_{i,j=1,\dots, d}$ be a strictly positive definite matrix in $\R^{d\times d}$ and
  \begin{alignat}{2}
    \label{eq:const-coeff-dirichlet}
    E_a({u}, {v})
    &= \sum_{i,j=1}^d \int_{\R^d} (D_i {u}(x)) a_{ij} (D_j{v}(x)) \; dx,
    &\qquad&
    u,v \in C_c^\infty(\R^d)
    ,
    \\
    \label{eq:const-coeff-dirichlet-discrete}
    E_a^*({u}, {v})
    &= \sum_{i,j=1}^d \sum_{x\in \Z^d} (\nabla_i {u}(x)) a_{ij} (\nabla_j {v}(x)),
    &\qquad&
    u,v\in C_c(\Z^d),
  \end{alignat}
  where $D_iu(x)$ is the partial derivative of ${u}(x)$ in direction $i=1,\dots,d$,
  \begin{equation}
    \nabla_i u(x) = u(x+e_i)-u(x)
  \end{equation}
  with $e_i$ the unit vector in the positive $i$th direction,
  and $C_c(\Z^d)$ is the space of functions $u:\Z^d \to \R$ with finite support.
  For $m^2 \geq 0$, further set
  \begin{equation}
    E_{a,\msq}(u,v) = E_a(u,v) + m^2 \int_{\R^d} u(x)v(x) \; dx
  \end{equation}
  and define $E_{a,\msq}^*$ analogously.
  Assume that the eigenvalues of $a$ are contained in the
  interval $[B_-^{\smash{2}}, B_+^{\smash{2}}]$, and in the discrete case
  also that $\msq \in [0, M_+^{\smash{2}}]$, for
  $B_-^{\smash{2}},B_+^{\smash{2}},M_+^{\smash{2}} > 0$;
  these assumptions are only important for uniformity in the constants below.
  
  In the continuous context, let $d$
  be the Euclidean distance on $X=\R^d$ and $\mu$ be the Lebesgue
  measure.
  It follows that $(X,\mu,d,E)$ satisfies \eqref{eq:prop-speed-gamma}
  with $\gamma=1$, $\theta(t) = B_+t$; 
  see Example~\ref{ex:systems} for more details.
  In the discrete context,
  let $d$ be the infinity distance on $X=\Z^d$, i.e., $d(x,y)=\max_{i=1,\dots d} |x_i-y_i|$,
  and $\mu$ be the counting measure.
  Then \smash{\eqref{eq:polynom-gamma}} holds with $B=B_+
  + M_+^{\smash{2}}$ and $\theta(n)=n$.

  Theorem~\ref{thm:decomposition} thus implies that the Green's functions
  associated to $E_{a,\msq}$ and $E_{a,\msq}^*$ admit finite range
  decompositions. Let us denote their kernels by 
  $\phi_t(x,y; a,\msq)$ and $\phi_t^*(x,y; a, \msq)$. 
  In addition to \eqref{eq:w-decay}, it is not difficult to obtain estimates
  on the decay of the derivatives of $\phi_t$ and $\phi_t^*$,
  like \eqref{eq:grad-s-decay}, in this situation of constant coefficients.
  Since these estimates are of interest for applications, we provide the details
  in Section~\ref{sec:constant-coefficients} (in a slightly more general context).
  We show that
  there are constants $C_{l,k}>0$ depending only
  on $B_-$ and $B_+$, and in the discrete case also on $M_+$, such
  that
  \begin{equation} \label{eq:const-coeff-w-est}
    |D_a^{l_a} D_{\msq}^{l_{\msq}} D_y^{\smash{l_y}} D_x^{l_x} \phi_t(x,y;a,m^2)|
    \leq C_{l,k} t^{-(d-2)-l_x-l_y+2l_{\msq}} (1+m^2t^2)^{-k} 
  \end{equation}
  and
  \begin{equation} \label{eq:const-coeff-w*-est}
    |D_a^{l_a} D_{\msq}^{l_{\msq}} \nabla_y^{\smash{l_y}} \nabla_x^{l_x} \phi_t^*(x,y,t;a,m^2)|
    \leq C_{l,k} t^{-(d-2)-l_x-l_y+2l_{\msq}} (1+m^2t^2)^{-k} 
  \end{equation}
  for all integers $l_a$, $l_{\msq}$, $l_x$, $l_y$, and $k$ such that 
  \begin{equation}
    l_{\msq} < \thalf(d + l_x+l_y),
  \end{equation}
  and that the following approximation result holds: There is $c>0$ such that
  \begin{equation} \label{eq:const-coeff-w-w*-approximation}
    \nabla_x^{l_x} \nabla_y^{\smash{l_y}} \phi_t^*(x,y; a, m^2)
    = D_x^{l_x} D_y^{\smash{l_y}} \phi_t(cx,cy; a,m^2)
    + O(t^{-(d-2)-l_x-l_y-1}(1+m^2t^2)^{-k})
    .
  \end{equation}
  In the discrete case, we have again assumed $t \geq 1$,
  whereas for $t<1$ the explicit formula $\phi_t^*(x,y,t;a,m^2)= C{\bf 1}_{x=y}t$ holds
  as stated below \eqref{eq:w-decay}.

  This reproduces and generalizes many results of \cite{MR2070102,AKM12a}.
  More precisely, we verify that there exists a smooth function
  $\bar \phi: \R^d \times [B_-^{\smash{2}},B_+^{\smash{2}}] \times
  [0,\infty) \to \R$ supported in $|x| \leq B_+$ such 
  that
  \begin{equation} \label{eq:const-coeff-w-hom}
    \phi_t(x,y;a,m^2) = t^{-(d-2)} \bar \phi(\frac{x-y}{t}; a,m^2t^2)
  \end{equation}
  which has the same structure as \eqref{eq:s-Newton} when $m^2=0$;
  this is scale invariance.
  Moreover, by \eqref{eq:const-coeff-w-w*-approximation},
  the discrete Green's function has a scaling limit
  and the error is of the order of the rescaled lattice spacing $O(t^{-1})$.
  This result improves \cite{BrydgesMitter12}.
\end{example}

\begin{example}[Elliptic operators and systems with variable coefficients] \label{ex:systems}
  Let $M \in \N$ and $a_{ij}: \R^d \to \R^{M\times M}$, $i,j=1,\dots, d$, be the smooth coefficients
  of a uniformly elliptic system (or in particular, if $M=1$, of a uniformly elliptic operator):
  \begin{equation}
    B_-^2 |{\bf \xi}|^2 \leq \sum_{k,l=1}^M \sum_{i,j=1}^d a_{ij}^{kl}(x) \xi_i^k \xi_j^l \leq B_+^2 |{\bf \xi}|^2
    \quad \text{for all ${\bf \xi} \in \R^{dM}$, $x \in \R^d$,}
  \end{equation}
  with $B_-,B_+ > 0$. Let us write ${\bf u} = ({\bf u}^1, \dots, {\bf u}^M) \in \R^{dM}$ with ${\bf u}^i \in \R^d$, $i=1, \dots, M$.
  Let
  \begin{equation} \label{eq:systems-dirichlet}
    E({\bf u}, {\bf v})
    = \sum_{i,j=1}^d \int_{\R^d} (D_i {\bf u}^k(x)) a_{ij}^{kl}(x) (D_j {\bf u}^l(x)) \; dx,
    \quad {\bf u}, {\bf v} \in C_c^\infty(\R^d, \R^M)
  \end{equation}
  and analogously in the discrete case
  (as in \eqref{eq:const-coeff-dirichlet}, \eqref{eq:const-coeff-dirichlet-discrete}).

  To apply Theorem~\ref{thm:decomposition}, $(X,\mu,d)$ is defined by
  $X = \R^d \times \{1, \dots, M\}$, $\mu$ is the product of the
  Lebesgue measure on $\R^d$ and the counting measure on $\{1, \dots, M\}$, and
  the distance is given by $d((x,i),(y,j)) = d(x,y)$. In particular,
  $d$ is only a pseudometric on $X$.
  We may use the identification of ${\bf u} : \R^d \to \R^M$
  and $u: X \to \R$ by $u(x,i) = {\bf u}^i(x)$.
  
  It suffices to verify the condition $(P_{1,B_+t})$ 
  for smooth, compactly supported ${\bf u}_0: \R^d \to \R^M$.
  For such a ${\bf u}_0$, set, by using spectral theory for self-adjoint operators:
  \begin{equation}
    {\bf u}(t) := \cos((L+m^2)^{\half}t){\bf u}_0.
  \end{equation}
  Then, since ${\bf u}_0$ is smooth,
  ${\bf u}(t,x): \R \times \R^d \to \R^M$
  is smooth jointly in $(t,x)$, and 
  %
  %
  \begin{equation} \label{eq:systems-wave-equation}
    \partial_t^2 {\bf u} + L{\bf u} + m^2 {\bf u} = 0, \quad \partial_t{\bf u}(0) = 0, \; {\bf u}(0) = {\bf u}_0
  \end{equation}
  holds in the classical sense.
  If $M=1$, $m^2=0$, and $a$ is the $d\times d$ identity matrix,
  $(P_{1,t})$ is the finite propagation speed of the wave equation.

  Similarly, in the general situation, the property $(P_{1,{B_+}t})$ can be deduced from
  the finite propagation speed of first order hyperbolic systems.
  This is well-known, but the explicit reduction
  for the case of \eqref{eq:systems-wave-equation} with \eqref{eq:systems-dirichlet} is difficult
  is to find in the literature. Let us therefore sketch how to convert \eqref{eq:systems-wave-equation}
  to a hyperbolic system for readers interested in this case.
  For example, one can define ${\bf v}: \R \times \R^d \to \R^{(d+2)M}$ by:
  \begin{equation}
    {\bf v}_{0}^k = \partial_{t} {\bf u}^k,
    \quad
    {\bf v}_{i}^k = \sum_{j=1}^d\sum_{l=1}^M a_{ij}^{kl} \partial_{x_j} {\bf u}^l,
    \quad
    {\bf v}_{d+1}^k = {\bf u}^k,
    \quad i=1, \dots, d, \; k=1, \dots, M,
  \end{equation}
  It follows that ${\bf v}$ satisfies
  \begin{equation}
    {\bf S}\partial_t{\bf v} + \sum_{j=1}^d {\bf A}_j \partial_{x_j} {\bf v} + {\bf B} {\bf v}= 0,
    \quad {\bf v}(0) = ({\bf 0}, (aD {\bf u}_0)^1, \dots, (aD {\bf u}_0)^d), {\bf u}_0)
  \end{equation}
  where ${\bf S}, {\bf A}_j, {\bf B}: \R^d \to R^{(d+2)M \times (d+2)M}$ are defined as the block matrices
  \begin{equation}
    {\bf S} = \begin{pmatrix}
      1_{M\times M} & 0_{dM\times M} & 0_{M\times M} \\
      0_{M\times dM} & a^{-1} & 0_{M\times dM} \\
      0_{M\times M} & 0_{dM\times M} & 1_{M\times M}
    \end{pmatrix},
    \quad
    {\bf B} = \begin{pmatrix}
      0_{1\times 1} & 0_{d\times 1} & m^2 \\
      0_{1 \times d}  & 0_{d\times d} & 0_{1\times d}  \\
      -m^2 & 0_{d\times 1} & 0_{1\times 1}
    \end{pmatrix} \otimes 1_{M\times M}
    ,
  \end{equation}
  and
  \begin{equation}
    {\bf A}_i = \begin{pmatrix}
      0 & -\delta_{1i} & \cdots & -\delta_{di} & 0\\
      -\delta_{1i} & 0 & \hdots & 0 & 0 \\
      \vdots & \vdots & \ddots & \vdots & 0 \\
      -\delta_{di} & 0 & \hdots  & 0 & 0\\
      0 & 0 & \cdots & 0 & 0
    \end{pmatrix} \otimes 1_{M\times M}
    ,\quad i=1, \dots, d
    .
  \end{equation}
  It is immediate that this system is symmetric uniformly hyperbolic, by the symmetry and uniform ellipticity of
  the matrix $a$. The property $(P_{1,B_+ t})$ now follows from the finite propagation speed of linear
  hyperbolic systems; see e.g.\ \cite{MR2524198,MR2273657}.

  Nash showed \cite{MR0100158} that $(H_{d,\omega})$ holds when $M=1$.
  In \cite{MR2091016,MR2425701}, conditions are given for $(H_{d,\omega})$ to hold when $M > 1$.
  In particular, this includes the constant coefficient case.
  The latter case can be treated by using the Fourier transform; see Section~\ref{sec:constant-coefficients}.
\end{example}

\begin{example}[Random walk on graphs] \label{ex:graphs}
  Let $(X,E)$ be a (locally finite) graph, with vertex set $X$ and edge set $E \subset P_2(X)$,
  where $X$ is a countable (or finite) set and $P_2(X)$ are the subsets of $X$ with two elements.
  Let $d : X \times X \to [0,\infty]$ be the graph distance on $(X,E)$, i.e., $d(x,y)$ is the (unweighted) length
  of the shortest path from $x$ to $y$.

  Suppose that edge weights $\mu_{xy}=\mu_{yx} \geq 0$, $x,y\in X$ are given.
  These induce a natural measure, also denoted $\mu$, on $X$ by:
  \begin{equation} \label{eq:graph-measure-defn}
    \mu_x = \sum_{y\in X} \mu_{xy}, \quad \mu(A) = \sum_{x\in A} \mu_x \quad \text{for all $A \subseteq X$}.
  \end{equation}
  The associated Dirichlet form is
  \begin{equation} \label{eq:graphs-Dirichlet}
    E(u,u) = \thalf \sum_{xy \in E} \mu_{xy} (u(x)-u(y))^2
    \quad \text{for all $u \in D(E) = L^2(\mu)$}
  \end{equation}
  and its generator is given by
  \begin{equation}
    Lu(x) = \mu_x^{-1} \sum_{y \in X} \mu_{xy} (u(x)-u(y))
    \quad \text{for all finitely supported $u: X \to \R$.}
  \end{equation}
  $L$ is called the \emph{probabilistic Laplace operator} associated
  to the simple random walk on the weighted graph $(X,\mu)$ with
  transition probabilities $\mu_{xy}/\mu_x$. Let us remark that a
  probabilistic interpretation (or a maximum principle) does not hold
  in general for
  Examples~\ref{ex:constant-coefficients}--\ref{ex:systems} (when $a$
  is non-diagonal or vector-valued).

  The Dirichlet form \eqref{eq:graphs-Dirichlet} is bounded on
  $L^2(\mu)$ with operator norm $2$ so that the property
  \smash{\eqref{eq:polynom-gamma}} holds with $\theta(n)=n$ and $B=2$,
  and Theorem~\ref{thm:decomposition} is applicable.

  For applications, it is often useful to add a killing rate to the
  random walk: The probabilistic Green density with killing rate
  $\kappa \in (0,1)$ is defined by:
  \begin{equation} \label{eq:graph-killed-Green}
    G^\kappa(x,y) = \sum_{n\geq 0} p^n(x,y) \kappa^n
    = (\kappa L + (1-\kappa))^{-1}(x,y) = (L^\kappa)^{-1}(x,y)
  \end{equation}
  where $p^n(x,y)$ is the kernel of the operator $P^n$ on $L^2(\mu)$.
  Note that \eqref{eq:graph-killed-Green} only converges for $\kappa=0$
  when the random walk is transient, but that $L^{-1}$ still makes sense
  as a quadratic form on its appropriate domain when the random walk is recurrent,
  as in \eqref{eq:Newton-form}, \eqref{eq:Newton-form-domain} for $d=2$.
  Note further that $\spec(L^\kappa) \subseteq [0,2]$ for all $\kappa \in
  [0,1]$, so that Theorem~\ref{thm:decomposition} is applicable
  uniformly in $\kappa \in [0,1]$.

  Closely related to the killed Green's function $G^\kappa$ is the resolvent kernel of $L$.
  The resolvent of $L$ is defined on $L^2(\mu)$ by $G_{\msq}= (L+m^2)^{-1}$ for $m^2 > 0$.
  It is related to the killed Green's density by:
  \begin{equation}
    G^\kappa = \kappa^{-1} G_{(1-\kappa)/\kappa}.
  \end{equation}
  One difference compared with the killed Green's function is that
  $L+m^2$ is not bound\-ed uniformly in $m^2 \geq 0$.  To achieve the
  condition \smash{\eqref{eq:polynom-gamma}} for fixed $B>0$, it is
  therefore necessary to restrict to $m^2 \leq M_+^2$ with $M_+^2 = B
  - 2$.
\end{example}  

\begin{remark}
  Other examples which Theorem \ref{thm:decomposition} is applicable to include
  Dirichlet spaces that satisfy a Davies-Gaffney estimate \cite{MR2114433} such as
  weighted manifolds and quadratic forms corresponding to powers of elliptic
  operators like $\Delta^2$.
\end{remark}

\subsection{Remarks}

\begin{remark}
  Theorem~\ref{thm:decomposition} also gives the decomposition into sums
  as in \cite{MR2240180,MR2070102,AKM12a}: Suppose that the
  assumptions of Theorem~\ref{thm:decomposition} are satisfied and,
  for notational simplicity, that the resulting decomposition has a
  kernel. Then, for any $L > 1$,
  \begin{equation}
    \Phi(x,y) = \sum_{j \in \Z} C_j(x,y)
    \quad \text{for all $x,y \in X \times X$}
  \end{equation}
  where the functions $C_j : X \times X \to [0,\infty)$, $j\in \Z$ are given by
  \begin{equation}
    C_j(x,y)
    := \int_{L^{j-1}}^{L^{j}} \phi_t(x,y) \; \frac{dt}{t}
    \quad \text{for all $x,y \in X$.}
  \end{equation}
  They satisfy the following properties:
  \begin{gather}
    C_j \text{ is the kernel of a positive semi-definite form},
    \\
    C_j(x,y) = 0
    \quad \text{for all $x,y \in X$ with $d(x,y) \geq L^j$,}
  \end{gather}
  and, if \eqref{eq:H-alpha} holds, 
  \begin{equation}
    |C_j(x,y)|
    \leq
    c_\alpha(x,y) \begin{cases}
      L^{-(\alpha-2)(j-1)} & (\alpha > 2)
      \\
      L^{(2-\alpha)j} & (\alpha < 2)
      \\
      \log(L) &(\alpha = 2)
    \end{cases}
  \end{equation}
  with $c_\alpha(x,y)$ is independent of $L$.  Thus, $(C_j)_{j\in\Z}$
  is a finite range decomposition into discrete scales of the Green's
  function $\Phi$.  Similarly, gradient estimates such as
  \eqref{eq:const-coeff-w-est}, \eqref{eq:const-coeff-w*-est},
  \eqref{eq:const-coeff-w-w*-approximation} in
  Example~\ref{ex:constant-coefficients} have obvious
  discrete versions.
\end{remark}

\begin{remark}
  More generally than in Theorem~\ref{thm:decomposition}, we may consider a \emph{family} of symmetric forms, $E^s$, $s \in Y$, where $Y$
  is a domain in a Banach space, with generators $L^s$. Let us assume
  that $E^s$ is smooth in $s$, in the following sense: There exists a projection-valued measure
  $P$ on a measurable space $M$ and a function $V: M \times Y \to (0,\infty)$, smooth in $Y$, such that
  \begin{equation} \label{eq:smooth-family}
    F(L^s)
    = \int_{\spec(L^s)} F(\lambda) \; dP_\lambda^s
    = \int_M F(V(s, \tau)) \; dP_\tau.
  \end{equation}
  An example of this condition is $E^{s}(f,f) = E(f,f) +
  s(f,f)$ in which case $V(s, \lambda) = \lambda + s$ and
  $(L^{s})^{-1}$ is the resolvent of $L$; similarly, the killed
  Green's function of Example~\ref{ex:graphs} can be expressed in this way.  Then the family of
  kernels $\phi^s$ is continuous in $s$, and if \eqref{eq:H-alpha}
  holds for $s=0$, and $V(\lambda,s) \geq z^2(s)
  V(\lambda,0) + m^2(s)$, then
  \begin{equation} \label{eq:decay-w-eta}
    |\phi^\eta_t(x,y)| \leq C_{\alpha,\gamma,l} \sqrt{\omega(x)\omega(y)} (z(s)t)^{-(\alpha-2)/\gamma}
    (1+tm(s))^{-l}
    .
  \end{equation}
  This can be verified by a straightforward adaption of the proof of
  Theorem~\ref{thm:decomposition}.
\end{remark}

\section{Proof of Theorem~\ref{thm:decomposition}}

\subsection{Spectral decomposition}

The starting point for the proof is the spectral representation
of the Green form \eqref{eq:G-L12}:
\begin{equation}
  \Phi(f,f)
  = \int_{\spec(L)} \lambda^{-1} \; d(f, P_\lambda f)
  \quad\text{for all $f\in D(\Phi)$}
  ,
\end{equation}
where $f \in D(\Phi)$ implies that the integral can be restricted to
$\spec(L) \setminus 0$.
The main result follows by decomposition of the
function $\lambda^{-1}: \spec(L)\setminus 0 \to \R_+$. 
Different decompositions are needed under
the two conditions \eqref{eq:prop-speed-gamma}, \eqref{eq:polynom-gamma}.
The main idea of the proof is that decompositions with good properties
exist. The result that we prove after using it to
deduce Theorem~\ref{thm:decomposition} is summarized in the following lemma.

\begin{lemma}[Spectral decomposition] \label{lem:W}
  Suppose that $L$ satisfies \eqref{eq:prop-speed-gamma} or
  \eqref{eq:polynom-gamma}.
  Then there exists a smooth family of functions $W_t
  \in C^\infty(\R)$, $t>0$, such that for all $\lambda \in \spec(L)
  \setminus 0$, $t>0$, and all integers $l$,
  \begin{equation} \label{eq:W-decomp}
    \lambda^{-1} = \int_0^\infty t^{\frac{2}{\gamma}} W_t(\lambda) \; \frac{dt}{t},
  \end{equation}
  \begin{equation} \label{eq:W-pos}
    W_t(\lambda) \geq 0,
  \end{equation}
  \begin{equation} \label{eq:W-decay}
    (1+t^{\frac{2}{\gamma}}\lambda)^l W_t(\lambda) \leq C_{l} ,
  \end{equation}
  and that for all $u \in C_c(X)$,
  \begin{equation} \label{eq:W-support}
    \supp(W_t(L)u) \subseteq N_{\theta(t)}(\supp(u))
    .
  \end{equation}
  In the discrete case \eqref{eq:polynom-gamma},
  we use the convention $\gamma=1$ and have assumed $t \geq 1$ for \eqref{eq:W-decay}.
  For $t<1$ one instead has the explicit formula $W_t(\lambda) = C/t$ for a constant $C>0$.
\end{lemma}

\begin{remark}
  More precisely, we will give explicit formulae for $W_t$ that imply
  (assuming $t\geq 1$ in the discrete case \eqref{eq:polynom-gamma})
  \begin{equation} \label{eq:W-decay-derivatives}
    (1+t^2\lambda)^l \lambda^{m} \left| \ddp{^m}{\lambda^m} W_t(\lambda)\right| \leq C_{l,m}
  \end{equation}
  for all $m$ and $l$, improving \eqref{eq:W-decay}.
  This improvement is used in Section~\ref{sec:constant-coefficients}.
\end{remark}

\begin{proof}[Proof of Theorem~\ref{thm:decomposition}]
  It follows from \eqref{eq:W-decomp} that, for any $f \in D(\Phi)$,
  \begin{align} \label{eq:G-decomp-proof}
    \Phi(f,f)
    &= \int_{\spec(L)} \left( \int_0^\infty t^{\frac{2}{\gamma}}  W_t(\lambda) \; \frac{dt}{t} \right) \; d(f,P_\lambda f)
    \\
    &= \int_{0}^\infty t^{\frac{2}{\gamma}} \left( \int_{\spec(L)}  W_t(\lambda) \; d(f,P_\lambda f) \right) \; \frac{dt}{t} \nonumber
    \\
    &= \int_{0}^\infty t^{\frac{2}{\gamma}} (f, W_t(L) f) \; \frac{dt}{t} \nonumber
    .
  \end{align}
  The exchange of the order of the two integrals in the equation above
  is justified by non-negativity of the integrand, by \eqref{eq:W-pos}.
  The latter also implies that $(f, W_t(L)f)
  \geq 0$ for all $f \in L^2(X)$.
  The polarization identity allows to recover $\Phi(f,g)$ for all $f,g \in D(\Phi)$.
  Finally, \eqref{eq:W-support} completes the verification of
  \eqref{eq:S-decomp} for $\Phi_t$ defined by
  \begin{equation}
    \Phi_t(f,g) = t^{\frac{2}{\gamma}} (f,W_t(L)g).
  \end{equation}

  In particular, for \eqref{eq:polynom-gamma} and $t<1$, 
  the formula $W_t(\lambda) = C/t$ and $\gamma=1$ imply $\Phi_t(f,g) = Ct(f,f)$.
  In the remaining cases,
  it remains to prove that \eqref{eq:H-alpha} implies \eqref{eq:w-decay}.
  The semigroup property and the continuity of $p_t$ imply that $p_t \in C_b(X,L^2(X))$ with
  \begin{gather}
    \label{eq:p-L2}
    \|p_t(x,\cdot)\|_{L^2(X)} = \int_X p_t(x,y)p_t(y,x) \; d\mu(y) = p_{2t}(x,x),
    \\
    \label{eq:p-L2-cont}
    \|p_t(x,\cdot)-p_t(y,\cdot)\|_{L^2(X)} = p_{2t}(x,x)+p_{2t}(y,y)-2p_{2t}(x,y) \to 0 \quad \text{as $x\to y$}.
  \end{gather}
  This implies that $e^{-tL}: L^2(X) \to C_b(X)$ is continuous (since $e^{-tL}f(x) = (p_t(x,\cdot), f)$).
  Duality then implies continuity of $e^{-tL}: C_b(X)^* \to L^2(X)$ (with respect to the
  strong topology on $C_b(X)^*$).
  Let $M(X) \subseteq C_b(X)^*$ be the space of signed finite Radon measures on $X$
  equipped with the weak-* topology. Let $m_i \in M(X)$ with $m_i \to 0$. Then:
  \begin{multline} \label{eq:p-seqcont}
    \|e^{-tL}m_i\|_{L^2(X)}
    = \biggl( \int_X \biggl( \int_X p_t(x,y) \; dm_i(y) \biggr)^2 \; d\mu(x) \biggr)^{\half}
    \\
    = \biggl( \int_X \int_X (p_t(y,\cdot), p_t(z,\cdot)) \; dm_i(y) \, dm_i(z) \biggr)^{\half}
    \to 0
  \end{multline}
  which means that $e^{-tL}:M(X) \to L^2(X)$ is continuous (because $X$ is separable and 
  therefore the weak-* topology of $M(X)$ is metrizable).
  This implies that $(1+t^{2/\gamma}L)^{-l}: M(X) \to L^2(X)$ is likewise continuous
  for all $l > \alpha/4$. To see this, we use the relation
  \begin{equation} \label{eq:Gamma-formula}
    (1+t^{2/\gamma}\lambda)^{-l} = \Gamma(l)^{-1} \int_0^\infty e^{-s} s^{l-1} e^{-st^{2/\gamma}\lambda} \; ds
  \end{equation}
  which holds by the change of variables formula and the
  definition of Euler's gamma function.
  The spectral theorem thus implies that, for any $u \in L^2(X)$,
  \begin{equation}
    \|(1+t^{2/\gamma}L)^{-l}u\|_{L^2(X)}
    \leq \Gamma(l)^{-1} \int_0^\infty e^{-s} s^{l-1} \|e^{-st^{2/\gamma} L} u\|_{L^2(X)}\; ds
    .
  \end{equation}
  Since $\mu$ has full support, $L^2(X) \cap M(X)$ 
  is dense in $M(X)$ (where $L^p(X)$ is always with respect to $\mu$),
  and the claimed continuity of
  $(1+t^{2/\gamma}L)^{-l}:M(X) \to L^2(X)$ follows from \eqref{eq:p-seqcont}.
  In particular, the pointwise bound for $p_t$ implies that for $l > \alpha/4$,
  \begin{align} \label{eq:heat-kernel-Gamma-bound}
    \|(1+t^{2/\gamma}L)^{-l}\delta_x\|_{L^2(X)}
    &\leq \Gamma(l)^{-1} \int_0^\infty e^{-s} s^{l-1} \|e^{-st^{2/\gamma} L} \delta_x\|_{L^2(X)}\; ds
    \\
    &\leq \Gamma(l)^{-1} \sqrt{\omega(x)} t^{-\alpha/2\gamma} \int_0^\infty e^{-s} s^{l-1-\alpha/4} \; ds
    \nonumber
    \\
    &= C \sqrt{\omega(x)} t^{-\alpha/2\gamma}
    \nonumber
    .
  \end{align}
  Let $\kappa_t(\lambda) = W_t(\lambda)^{1/2}$.
  Then \eqref{eq:W-decay} and the spectral theorem 
  also imply that
  \begin{equation}
    \|\kappa_t(L)(1+t^{2/\gamma}L)^l\|_{L^2(X)\to L^2(X)}
    = \sup_{\lambda > 0} \kappa_t(\lambda)(1+t^{2/\gamma}\lambda)^l
    \leq C_l,
  \end{equation}
  uniformly in $t>0$.
  It follows from \eqref{eq:heat-kernel-Gamma-bound} that $\kappa_t(L): M(X) \to L^2(X)$
  with
  \begin{equation} \label{eq:kappa-delta-bound}
    \|\kappa_t(L)\delta_x\|_{L^2} \leq C \sqrt{\omega(x)} t^{-\alpha/2\gamma}
    .
  \end{equation}
  Finally, by the Cauchy-Schwarz inequality,
  \begin{equation}
    |\phi_t(x,y)|
    =
    t^{2/\gamma} (\kappa_t(L) \delta_y, \kappa_t(L)\delta_x)
    \leq
    t^{2/\gamma} \|\kappa_t(L)\delta_y\|_{L^2(X)} \|\kappa_t(L)\delta_x\|_{L^2(X)}
  \end{equation}
  which,  with \eqref{eq:kappa-delta-bound},  proves \eqref{eq:w-decay}.
  The continuity of $\phi_t$ is implied by the continuity of
  $\kappa_t(L): M(X) \to L^2(X)$ and of $\delta_x$ in $x \in X$ (in the weak-* topology).
\end{proof}

\begin{remark}
  The decay for $\phi^\eta$ claimed in \eqref{eq:decay-w-eta} can be obtained by a straightforward
  generalization of the above argument, replacing \eqref{eq:Gamma-formula} by
  \begin{equation}
    (1+t^{2/\gamma}z^2\lambda + t^{2/\gamma}m^2)^{-l}
    = \Gamma(l)^{-1} \int_0^\infty e^{-s} s^{l-1} e^{-st^{2/\gamma}m^2} e^{-s z^2 t^{2/\gamma} \lambda}  \; ds
    .
  \end{equation}
\end{remark}

\begin{remark}
  Furthermore, by \eqref{eq:W-decay}, the operators $W_t(L)$ are
  smoothing for $t>0$, in the general sense that, for any $t>0$,
  \begin{equation}
    W_t(L): L^2(X) \to C^\infty(L),
    \quad 
    \text{where }
    C^\infty(L) := \bigcap_{n=0}^\infty D(L^n)
    \subset L^2(X)
  \end{equation}
  is the set of $C^\infty$-vectors for $L$; see \cite{MR0493420}.
  Standard elliptic regularity estimates imply e.g.\ that $C^\infty(L) = C^\infty(X)$
  when $E$ is the quadratic form associated to an elliptic operator with smooth coefficients.
\end{remark}

\subsection{Proof of Lemma~\ref{lem:W}}

To complete the proof of Theorem~\ref{thm:decomposition}, it remains
to demonstrate Lemma~\ref{lem:W}.
We first prove it under condition \eqref{eq:prop-speed-gamma} in Lemma~\ref{lem:W-cont} below;
this proof is quite straightforward using the assumption and \eqref{eq:Newton-decomp}.
Subsequently, we prove Lemma~\ref{lem:W} in the situation of condition \eqref{eq:polynom-gamma}
in Lemma~\ref{lem:W-discr}; here additional ideas are required.

To fix  conventions, let us define the Fourier transform of an integrable function $\varphi: \R \to \R$ by
\begin{equation}
  \hat \varphi(k) = (2\pi)^{-1} \int_{\R} \varphi(x) e^{-ikx} \; dx \quad \text{for all $k \in \R$}
  .
\end{equation}

\begin{lemma}[Lemma~\ref{lem:W} under \eqref{eq:prop-speed-gamma}] \label{lem:W-cont}
  For any $\varphi: \R \to [0,\infty)$ such that
  $\hat\varphi$ is smooth and symmetric with $\supp(\hat\varphi) \subseteq [-1,1]$,
  and for any $\gamma > 0$, there is $C>0$ such that
  \begin{equation} \label{eq:W-cont}
    W_t(\lambda) := C \varphi(\lambda^{\half \gamma}t)
  \end{equation}
  satisfies \eqref{eq:W-decomp}, \eqref{eq:W-pos}, \eqref{eq:W-decay},
  and also \eqref{eq:W-decay-derivatives}, for all $\lambda>0$, $t>0$;
  and if \eqref{eq:prop-speed-gamma} holds, then $(W_t)$ also satisfies \eqref{eq:W-support}.
\end{lemma}

\begin{remark}
  It is not difficult to see that such $\varphi$ exist. For example,
  if $\hat\kappa$ is a smooth real-valued function with support in $[-\half,\half]$, then
  $\varphi = |\kappa|^2$ satisfies the assumptions. 
  For simplicity, let us assume sometimes in the following that $\varphi$ is chosen such that
  $C=1$ when Lemma~\ref{lem:W} is applied.
\end{remark}

\begin{proof}
  Note that for any $\varphi:[0,\infty)\to\R$ with $t\varphi(t)$ integrable,
  there is $C > 0$ such that
  \begin{equation} \label{eq:1/lambda-decomp}
    \lambda^{-1}
    = C \int_0^\infty t^{\frac{2}{\gamma}} \varphi(\lambda^{\half \gamma} t) \; \frac{dt}{t}
    \quad \text{for all $\lambda>0$.}
  \end{equation}
  This simply follows (as in \eqref{eq:Newton-decomp}) because the right-hand side is homogeneous in $\lambda$ of degree $-1$, which is
  immediate by rescaling of the integration variable. This shows \eqref{eq:W-decomp};
  \eqref{eq:W-pos} is obvious by assumption; and \eqref{eq:W-decay} follows since $\hat\varphi$ is smooth.
  The improved estimate \eqref{eq:W-decay-derivatives} follows from the chain rule (or Fa\`a di Bruno's formula) and 
  \begin{equation} \label{eq:sqrt-lambda-derivative}
    \lambda^{m-\half\gamma} \left| \ddp{^m}{\lambda^m} \lambda^{\half\gamma} \right|
    \leq C_{\gamma,m}
  \end{equation}
  for non-negative integers $m$, using that $\supp(\hat\varphi) \subseteq [-1,1]$ implies that $\varphi$ is smooth.
  Moreover, since $\supp(\hat\varphi) \subset [-1,1]$, and since $\hat\varphi$ is smooth,
  \begin{equation}
    W_t(L)u = C \int_{-1}^1 \hat\varphi(s) \cos(L^{\half \gamma} ts)u \; ds
    \quad \text{for all $u \in L^2(X)$,}
  \end{equation}
  where the integral is the Riemann integral, i.e., the strong limit of its Riemann sums (with values in $L^2$).
  Therefore \eqref{eq:W-support} follows from \eqref{eq:prop-speed-gamma}.
\end{proof}

The previous proof makes essential use of the finite propagation speed
of the wave equation \eqref{eq:prop-speed-gamma} to prove
\eqref{eq:W-support}.  This property fails for discrete Dirichlet
forms such as \eqref{eq:const-coeff-dirichlet-discrete} where we
instead know the property \smash{\eqref{eq:polynom-gamma}} that
polynomials of degree $n$ of the generator have finite range
$\theta(n)$. 

This leads to the following problem. Find polynomials $W_t^*$, $t>0$,
of degree at most $t$ satisfying the properties
\eqref{eq:W-pos}, \eqref{eq:W-decay}, \eqref{eq:W-decay-derivatives} such that the
decomposition formula \eqref{eq:W-decomp} for $1/\lambda$ holds.  In
the proof of Lemma~\ref{lem:W-cont}, the verification of \eqref{eq:W-decay} (and
\eqref{eq:W-decay-derivatives}) and of the decomposition
formula \eqref{eq:W-decomp} are directly linked to the
``ballistic'' scaling of the wave equation: $W_t(\lambda)=W_1(\lambda t^2)$.
To construct polynomials satisfying such ``ballistic'' estimates,
we are led by the following remarkable discovery of Carne \cite{MR837740}:
The Chebyshev polynomials $T_k$, $k \in \Z$, defined by
\begin{equation} \label{eq:Chebyshev}
  T_k(\theta) = \cos(k\arccos(\theta))
  \quad \text{for all $\theta \in [-1,1]$, $k \in \Z$,}
\end{equation}
are solutions to the discrete (in space and time) wave equation in the following sense:
Let $\nabla_+f(n) = f(n+1)-f(n)$ and $\nabla_-f(n) = f(n-1)-f(n)$ be the discrete (forward and backward) time differences. 
Then, as polynomials in $X$,
\begin{equation}
  \nabla_-\nabla_+ T_n(X) = \nabla_+\nabla_- T_n(X) = 2(X-1) T_n(X).
\end{equation}
In particular, when $2(X-1)=-L$ or equivalently $X = 1-\half L$, then $v(n,x) = [T_n(1-\half L) u](x)$ solves
the following ``Cauchy problem'' for the discrete wave equation:
\begin{equation}
  - \nabla_+\nabla_- v + Lv = 0, \quad v(0) = u, (\nabla_+v-\nabla_-v)(0)=0
  .
\end{equation}
The analogy between the discrete- and the continuous-time wave equations is like that 
between the discrete- and the continuous-time random walk. 
It turns out that the structure of Chebyshev polynomials 
allows to prove the following lemma.

\begin{lemma}[Lemma~\ref{lem:W} under \eqref{eq:polynom-gamma}] \label{lem:W-discr}
  Let $\varphi: \R \to [0,\infty)$ satisfy the assumptions of Lemma~\ref{lem:W-cont}. Then
  $W_t^*: [0,4] \to [0,\infty)$, defined by
  \begin{equation} \label{eq:W-discr}
    W_t^*(\lambda) := \sum_{n \in \Z} \varphi(\arccos(1-\thalf \lambda) t-2\pi nt)
    \quad \text{for all $\lambda \in[0,4]$, $t>0$},
  \end{equation}
  is the restriction of a polynomial in $\lambda$ of degree at most $t$ to $[0,4]$,
  with coefficients smooth in $t$,
  and, for any $\varepsilon > 0$,
  \eqref{eq:W-decomp}, \eqref{eq:W-pos}, \eqref{eq:W-decay}, \eqref{eq:W-support}, 
  and \eqref{eq:W-decay-derivatives} hold for all $\lambda \in (0,4-\varepsilon]$, $t > 0$.
\end{lemma}


\begin{proof}
  The proof verifies that $W_t^*$ as defined in \eqref{eq:W-discr} has the asserted properties.
  Let
  \begin{equation} \label{eq:varphi*-defn}
    \varphi^*_t(x)
    := \sum_{n \in \Z} \varphi(xt-2\pi nt)
    = \sum_{k \in \Z} t^{-1}\hat\varphi(k/t) \cos(kx)
  \end{equation}
  where the second equality follows by symmetry of $\hat\varphi$,
  the change of variables formula,
  and a version of the Poisson summation formula which is easily verified,
  for sufficiently nice $\varphi$.
  Then the claim \eqref{eq:W-decomp} can be expressed as 
  \begin{equation} \label{eq:1/lambda-discr-decomp}
    \lambda^{-1} = \int_0^\infty t^{2} \varphi^*_t(\arccos(1-\thalf\lambda)) \; \frac{dt}{t}
    \quad \text{for all $\lambda \in (0,4]$}
    .
  \end{equation}

  Let $x = \arccos(1-\half \lambda)$ or equivalently $\lambda = 2(1-\cos x) =
  4 \sin^2(\frac{1}{2} x)$.  In terms of this change of variables,
  \eqref{eq:1/lambda-discr-decomp} and thus the claim
  \eqref{eq:W-discr} are then equivalent to
  \begin{equation} \label{eq:sin2-decomp}
    \tquarter
    \sin^{-2}(\thalf x)
    = \int_0^\infty t^{2} \varphi^*_t(x) \; \frac{dt}{t}
    \quad \text{for all $x \in (0,\pi]$.}
  \end{equation}
  The left-hand side defines a meromorphic function on $\C$ with poles at $2\pi \Z$.
  Its development into partial fractions is (see e.g.\ \cite[page 204]{MR510197})
  \begin{equation}
    \tquarter
    \sin^{-2} (\thalf x) = \sum_{n \in \Z} (x-2\pi n)^{-2}
    \quad \text{for all $x \in \C \setminus 2\pi\Z$}.
  \end{equation}
  It follows, by \eqref{eq:1/lambda-decomp} with $\gamma=1$ and $\lambda = (x-2\pi n)^{2}$, assuming $C=1$, that
  \begin{equation}
    \tquarter
    \sin^{-2} (\thalf x) = \sum_{n \in \Z} \int_0^\infty t^{2}  \varphi((x-2\pi n)t) \; \frac{dt}{t}
    \quad \text{for all $x \in (0,\pi]$}.
  \end{equation}
  The order of the sum and the integral can be exchanged, by non-negativity of the integrand,
  thus showing \eqref{eq:sin2-decomp} and therefore \eqref{eq:W-decomp}.

  To verify that $W_t^*$ is the restriction of a polynomial, we note
  that by \eqref{eq:W-discr}, \eqref{eq:varphi*-defn},
  and $\supp(\hat\varphi) \subseteq [-1,1]$,
  \begin{align}
    W_t^*(\lambda)
    =
    \varphi^*_t(\arccos(1-\thalf \lambda))
    &= \sum_{k\in\Z} t^{-1} \hat\varphi(k/t) \cos(k\arccos(1-\thalf \lambda))
    \\
    &= \sum_{k \in \Z\cap[-t,t]} t^{-1} \hat\varphi(k/t) T_k(1-\thalf \lambda)
    \nonumber
  \end{align}
  where $T_k, k \in \Z$, are the Chebyshev polynomials defined by
  \eqref{eq:Chebyshev}.  This shows that $W_t^*(\lambda)$ is indeed
  the restriction of a polynomial in $\lambda$ of degree at most $t$ to the
  interval $\lambda \in [0,4]$.  In particular, \eqref{eq:W-support}
  is a trivial consequence of \smash{\eqref{eq:polynom-gamma}} which states
  that polynomials in $L$ of degree $n$ have range at most $\theta(n)$.
  Also, for $t < 1$,
  \begin{equation}
    W_t^*(\lambda) = t^{-1} \hat\varphi(0) = Ct^{-1}.
  \end{equation}

  Finally, we verify the estimate \eqref{eq:W-decay-derivatives} and
  thus in particular \eqref{eq:W-decay}, for $t \geq 1$.
  To this end, we note that, in analogy to \eqref{eq:sqrt-lambda-derivative},
  for $\lambda \in [0,4-\varepsilon]$ and non-negative integers $m$,
  \begin{equation}
    \lambda^{m-\half} \left| \ddp{^m}{\lambda^m}\arccos(1-\thalf \lambda)\right| \leq C_{\varepsilon,m}
    .
  \end{equation}
  For example, for $m=1$,
  \begin{equation}
    \ddp{}{\lambda} \arccos(1-\thalf \lambda) = \thalf (\lambda-\tquarter \lambda^2)^{-\half}
    \leq \varepsilon^{-\half} \lambda^{-\half} \quad \text{for $\lambda \in [0,4-\varepsilon]$}
    .
  \end{equation}
  Therefore \eqref{eq:W-decay-derivatives} follows, by the chain rule (or Fa\`a di Bruno's formula), from
  \begin{equation} \label{eq:varphi*-decay}
    (1+t^2(1-\cos(x))^l t^{-m} \left| \ddp{^{m}}{x^{m}} \varphi^*_t(x)\right|  \leq C_{l,m},
  \end{equation}
  which holds for $t \geq 1$,
  as we will now show. The argument is essentially a discrete
  version of the classic fact that the Fourier transform acts
  continuously on the Schwartz space of smooth and rapidly decaying
  functions on $\R$.  To show \eqref{eq:varphi*-decay}, first note
  that
  \begin{equation}
    (1-\cos(x))e^{ikx}
    = e^{ikx}-\thalf e^{i(k+1)x} -\thalf e^{i(k-1)x}
    =:  \Delta_ke^{ikx}
  \end{equation}
  and thus by induction, for any $l \in \N$,
  \begin{equation}
    (1-\cos(x))^le^{ikx}
    = (1-\cos(x))^{l-1} \Delta_ke^{ikx}
    = \Delta_k (1-\cos(x))^{l-1} e^{ikx}
    = \Delta_k^l e^{ikx}
    .
  \end{equation}
  It follows by \eqref{eq:varphi*-defn} and summation by parts that
  \begin{align} \label{eq:varphi*-sumbyparts}
    (1+t^2(1-\cos(x))^l t^{-m} \ddp{^{m}}{x^{m}} \varphi_t^*(x)
    &=\sum_{k\in\Z} t^{-1} \hat\varphi(k/t)(ik/t)^{m} [(1 + t^2 \Delta_k)^l e^{ikx}]
    \\
    &=\sum_{k\in\Z} [(1 + t^2 \Delta_k)^l t^{-1} \hat\varphi(k/t) (ik/t)^{m}] e^{ikx}
    \nonumber
    .
  \end{align}
  Let $h(s) = \half (|s|-1) 1_{|s|\leq 1}$ for $s \in \R$.
  Then, for any smooth $f: \R  \to \R$,
  \begin{equation} \label{eq:Delta-conv}
    \Delta_k^nf(k) = (h^{*n} * D^{2n}f)(k), 
  \end{equation}
  where $*$ denotes convolution of two functions on $\R$,
  $h^{*n} = h * h * \dots * h$, and $Df$ is the derivative of $f$.
  Indeed,
  \begin{align}
    \Delta_k f(k)
    &= -\thalf \int_0^1 [Df(k+t) - Df(k-t)] \; dt
    \\
    &= -\thalf \int_0^1 \int_{-t}^t D^2f(k+s) \; ds \; dt
    = \int_{\R} D^2f(s) h(s-k) \; ds = (h * D^2f)(k)
    ,
    \nonumber
  \end{align}
  and
  \eqref{eq:Delta-conv} then follows by induction:
  \begin{equation}
    \Delta^{n+1}_{k}f
    = \Delta (h^{*n} * D^{2n} f)
    = h * D^2(h^{*n} * D^{2n} f)
    = h * h^{*n} * D^2D^{2n} f.
  \end{equation}
  It then follows using the facts that
  $\sum_{k\in\Z} |h^{*n}(k-s)| \leq C_n$, uniformly in $s \in \R$,
  and that $\hat\varphi$ is smooth and of rapid decay, for $t \geq 1$,
  \begin{align}
    & t^{-1} \sum_{k\in\Z} \Big| (1 + t^2\Delta_k^2)^l [\hat\varphi(k/t) (ik/t)^{m}] \Big|
    \\
    &= \sum_{n=0}^l C_{l,n} t^{-1} \sum_{k\in\Z} \int_{\R} |h^{*n}(k-s)|\, |[D^{2n}((\cdot)^{m} \hat\varphi)](s/t)| \; ds
    \nonumber
    \\
    &\leq Ct^{-1}+ \sum_{n=1}^l C_{l,n} t^{-1} \int_{\R} |[D^{2n}((\cdot)^{m}\hat\varphi)](s/t)| \; ds
    \nonumber
    \\
    &= Ct^{-1}+ \sum_{n=1}^l C_{l,n} \int_{\R} |[D^{2n}((\cdot)^{m}\hat\varphi)](s)| \; ds
    \leq C_{m,l}
    \nonumber
  \end{align}
  and thus \eqref{eq:varphi*-decay}, and therefore \eqref{eq:W-decay-derivatives},
  follow from this inequality and \eqref{eq:varphi*-sumbyparts}.
\end{proof}

\begin{proof}[Proof of Lemma~\ref{lem:W}]
  Lemma~\ref{lem:W} under \smash{\eqref{eq:prop-speed-gamma}} follows
  immediately from Lemma~\ref{lem:W-cont}; under
  \eqref{eq:polynom-gamma}, it follows from Lemma~\ref{lem:W-discr} by
  setting $W_t(\lambda) = W_t^*(\frac{3}{B} \lambda)$.
\end{proof}

\section{Extensions}

\subsection{Discrete approximation}

In view of the discussion about Chebyshev polynomials before
Lemma~\ref{lem:W-discr}, it is not surprising that the functions
$W_t^*$ of Lemma~\ref{lem:W-discr} approximate the $W_t$
of Lemma~\ref{lem:W-cont}.
In Proposition~\ref{prop:W-approximation} below, we show that this is indeed
the case with natural error $O(t^{-1})$ as $t\to\infty$.  This result is used
in Section~\ref{sec:constant-coefficients} to prove
\eqref{eq:const-coeff-w-w*-approximation}.

\begin{proposition}[Discrete
  approximation] \label{prop:W-approximation} Let $\varphi$ be as in
  Lemma~\ref{lem:W-cont} and \ref{lem:W-discr}, with associated
  functions $W_t$ and $W_t^*$ for $\gamma=1$.  Then, for any integer $l$,
  \begin{equation} \label{eq:W-approximation}
    |W_t^*(\lambda) - W_t(\lambda)| \leq C_l (1\vee t)^{-1} (1+t^2 \lambda)^{-l}
    \quad \text{for all $\lambda \in [0,4]$.}
  \end{equation}
  In particular, $W_t^*(\lambda/t^2) \to C \varphi(\lambda^{\half})$ as $t\to\infty$.
\end{proposition}

\begin{proof}
  Note that it suffices to restrict to $t\geq 1$, since for $t \leq 1$, the claim follows from
  \eqref{eq:W-decay}.
  The left-hand side of \eqref{eq:W-approximation}
  is then proportional to the absolute value of
  \begin{equation} \label{eq:W-approximation-rewritten}
    \varphi(\arccos(1-\thalf \lambda)t)-\varphi(\lambda^\half t)
    + \sum_{n\in \Z\setminus \{0\}} \varphi(\arccos(1-\thalf \lambda)t + 2\pi n t)
    .
  \end{equation}
  We estimate the difference of the first two terms in \eqref{eq:W-approximation-rewritten}
  and the sum separately, and show that each of them satisfies \eqref{eq:W-approximation}.
  The first two terms can be written as
  \begin{equation}
    \varphi(\arccos(1-\thalf \lambda)t)-\varphi(\lambda^\half t)
    = (\arccos(1-\thalf \lambda)-\lambda^\half )t \zeta_t(\lambda)
  \end{equation}
  with
  \begin{equation}
    \zeta_t(\lambda) 
    =
    \int_0^1 \varphi'(s \arccos(1-\thalf \lambda)t + (1-s) \lambda^\half t) \; ds.
  \end{equation}
  The bounds
  \begin{gather} \label{eq:sqrt-arccos-O}
    \sqrt{2\lambda} = \arccos(1-\lambda) + O(\lambda)
    \quad \text{as $\lambda \to 0+$}
    ,
    \\
    \label{eq:sqrt-arccos-bound}
    \sqrt{2\lambda} \leq \arccos(1-\lambda)\leq \tfrac{\pi}{2} \sqrt{2\lambda}
    \quad \text{for all $\lambda \in [0,2]$}
    ,
  \end{gather}
  and the rapid decay of $\varphi'$ therefore imply that
  \begin{equation}
    |\zeta_t(\lambda)| \leq C_l (1+\lambda t^2)^{-l}
  \end{equation}
  and
  \begin{equation} \label{varphi*-varphi-main-term-estimate}
    \varphi(\arccos(1-\thalf \lambda)t)-\varphi(\lambda^\half t)
    \leq C_l t^{-1} (1+t^2\lambda)^{-l}.
  \end{equation}
  
  To estimate the sum in \eqref{eq:W-approximation-rewritten},
  we can use the rapid decay of $\varphi$ with the inequality
  $x+y \geq 2 (xy)^{1/2}$ to obtain that
  \begin{align}
    \sum_{n\in \Z\setminus \{0\}} \varphi(xt + 2\pi n t)
    &\leq
    C_l \sum_{n\in \Z\setminus \{0\}} (1+xt + 2\pi n t)^{-l}
    \\
    &\leq
    C_l  (1+xt)^{-l/2} t^{-l/2} \sum_{n>0} n^{-l/2}
    \nonumber
    \leq
    C_l  (1+xt)^{-l/2} t^{-l/2}
    \nonumber
  \end{align}
  for any $l > 2$, with the constant changing from line to line.
  In particular, upon substituting $x=\arccos(1-\thalf \lambda)$,
  this bound and \eqref{eq:sqrt-arccos-bound} imply
  \begin{equation} \label{eq:varphi*-sum-estimate}
    \sum_{n\in \Z\setminus \{0\}} \varphi(\arccos(1-\thalf\lambda)t + 2\pi n t)
    \leq C_l t^{-2l} (1+t^2\lambda)^{-l}
    .
  \end{equation}
  The claim then follows by adding \eqref{varphi*-varphi-main-term-estimate} and \eqref{eq:varphi*-sum-estimate}.
\end{proof}

\subsection{Estimates for systems with constant coefficients}
\label{sec:constant-coefficients}

In this section, we verify the assertions of Example~\ref{ex:constant-coefficients}.
We work in the slightly more general context of second-order elliptic systems (instead of operators)
with constant coefficients. These are defined as in Example~\ref{ex:systems}, 
and the claims of Example~\ref{ex:constant-coefficients}
hold mutadis mutandis.
The analysis is straightforward, with aid of the Fourier transform.
It reproduces results of \cite{AKM12a}.

\subsubsection{Spectral measures}

The spectral measures corresponding to the vector-valued case of \eqref{eq:const-coeff-dirichlet} are
given in terms of the Fourier transform as follows. For $F: [0,\infty) \to \R$,
\begin{equation}
  (v, F(L)u)
  =
  \sum_{k,l=1}^M \int_{\R^{d}} \left[F\left( \sum_{i,j=1}^d a_{ij} \xi_i\xi_j\right)\right]_{kl}
  \; \smash{\overline{\hat v}}^k(\xi) \hat u^l(\xi) \; d\xi
\end{equation}
where $\hat u = (\hat u^1, \dots, \hat u^M)$ is the component-wise Fourier transform of $u = (u^1, \dots, u^M)$,
\begin{equation}
  a(\xi)
  := \sum_{i,j=1}^d a_{ij} \xi_i\xi_j
  = \left(\sum_{i,j=1}^d a_{ij}^{kl} \xi_i\xi_j\right)_{k,l=1,\dots,M}
\end{equation}
are symmetric positive definite $M \times M$ matrices, for all $\xi \in \R^d$,
and the matrices $F(a(\xi))$ are defined in terms of the spectral decomposition of $a(\xi)$.
Similarly, for the (vector-valued case of the) discrete Dirichlet form \eqref{eq:const-coeff-dirichlet-discrete},
\begin{equation}
  (v, F(L)u)
  = \sum_{k,l=1}^M \int_{[-\pi,\pi]^{d}} \left[F\left(\sum_{i,j=1}^d a_{ij} (1-e^{i\xi_i})(1-e^{-i\xi_j})\right)\right]_{kl}
  \; \smash{\overline{\hat v}}^k(\xi) \hat u^l(\xi) \; d\xi
\end{equation}
where here $\hat u$ is the component-wise discrete Fourier transform. Let us also write
\begin{equation}
  a^*(\xi)
  := \sum_{i,j=1}^d a_{ij} (1-e^{i\xi_i})(1-e^{-i\xi_j})
  = \left(\sum_{i,j=1}^d a_{ij}^{kl} (1-e^{i\xi_i})(1-e^{-i\xi_j})  \right)_{k,l=1,\dots,M}
  .
\end{equation}
We will often use, without mentioning this further,
that the spectra of $a(\xi)$ and $a^*(\xi)$ are bounded from above and from below by $|\xi|^2$.

\subsubsection{Estimates}

Let us introduce the following notation for derivatives:
For a function $u : \R^d \to \R$, we regard the $l$th derivative,
$D^lu(x)$, as an $l$-linear form, 
and $|D^lu(x)|$ is a norm of the form $D^lu(x)$. In terms of the Fourier
transform, we denote by $\smash{\hat D^{l}}(\xi)$ the
corresponding ``multiplier'' operator from functions to $l$-linear
forms, and by $|\smash{\hat D^l}(\xi)|$ its norm. Similarly, for a
discrete function $u: \Z^d \to \R$, the $l$th order discrete difference in
\emph{positive coordinate direction} is
denoted by $\nabla^lu(x)$ and has Fourier multiplier
$\smash{\hat \nabla^l(\xi)}$.
In particular, when $l=1$,
\begin{equation} \label{eq:hatD-hatnabla}
  \hat D(\xi) \cong (i\xi_1, \dots, i\xi_d),
  \quad
  \hat \nabla(\xi) \cong (e^{i\xi_1}-1, \dots, e^{i\xi_d}-1).
\end{equation}
Furthermore, $k$ and $p$ will denote integers that may be chosen arbitrarily,
and $C$ constants that can change from instance to instance and may depend on $k$ and $p$,
as well as $l=(l_x,l_y,l_a,l_{\msq})$, $B_+$, $B_-$, and $M_+$, but not on $x$, $\xi$, and $m$.

\begin{proof}[Proof of \eqref{eq:const-coeff-w-hom},\eqref{eq:const-coeff-w-est},\eqref{eq:const-coeff-w*-est}]
  We may assume that $t \geq 1$.
  It follows by the change of variables $\xi \mapsto t\xi$, from the fact that
  $a(\xi)$ is homogeneous of degree~$2$, and from $W_t(\lambda) = W_1(\lambda t^2)$ that
  \begin{align} \label{eq:w-wbar}
    \phi_t(x,y; a, m^2)
    &= t^2 \int_{\R^d} W_t(a(\xi) +m^2) e^{i (x-y) \cdot \xi}\; d\xi
    \\
    &= t^{-(d-2)} \bar \phi (\frac{x-y}{t}; a, m^2t^2)
    \nonumber
  \end{align}
  with
  \begin{equation} \label{eq:wbar-formula}
    \bar \phi(x; a, m^2) := \int_{\R^d} W_1(a(\xi) +m^2) e^{i (x-y) \cdot \xi}\; d\xi
  \end{equation}
  which is supported in $|x|\leq B_+$. This verifies \eqref{eq:const-coeff-w-hom}.
  Furthermore, \eqref{eq:const-coeff-w-est} is a straightforward
  consequence of \eqref{eq:w-wbar} by differentiation and
  \eqref{eq:W-decay-derivatives}.  Let us omit the details and only
  verify them explicitly in the discrete case
  \eqref{eq:const-coeff-w*-est}:
  The (derivatives of the) decomposition kernel $\phi_t^*$ can here be expressed as
  \begin{equation}
    D_a^{l_a} D_{\msq}^{l_{\msq}}\nabla_x^{l_x} \nabla_y^{\smash{l_y}} \phi_t^*(x,y; a, \msq)
    = t^{-(d-2)-l_x-l_y+2l_{\msq}} \bar \phi_{t;l}^*(x-y; a,\msq)
  \end{equation}
  with
  \begin{equation}
    \bar \phi_{t;l}^*(x; a,m^2)
    = t^{d+l_x+l_y-2l_{\msq}} \int_{[-\pi,\pi]^d} D_a^{l_a} D_{\msq}^{l_{\msq}}W_t^*(a^*(\xi) +m^2)
    \smash{\overline{\hat \nabla}}^{l_y} \hat\nabla^{l_x} e^{i x\cdot \xi} \; d\xi
    .
  \end{equation}
  Thus \eqref{eq:W-decay-derivatives}, $|\smash{\hat \nabla}(\xi)| \leq C|\xi|$,
  and $\eta \cdot a^*(\xi)\eta  \geq C|\xi|^2 |\eta|^2$ for $\eta \in \R^M$ imply
  \begin{align}
    |\bar \phi_{t;l}^*(x; a,m^2)|
    &\leq
    C \int_{[-\pi,\pi]^d} (1+C|\xi|^2 t^2 +m^2 t^2)^{-k-p} (t|\xi|)^{l_x+l_y-2l_{\msq}} \; t^d d\xi
    \\
    &\leq 
    C (1+m^2 t^2)^{-k} \int_{\R^d} (1+C|\xi|^2)^{-p} |\xi|^{l_x+l_y-2l_{\msq}} \; d\xi
    \nonumber
  \end{align}
  and therefore that the integral converges if $\thalf(d+l_x+l_y) > l_{\msq}$
  and $p$ is chosen sufficiently large. It follows that
  \begin{equation}
    |\bar \phi_{t;l}^*(x; a,m^2)|
    \leq
    C (1+m^2t^2)^{-k}
  \end{equation}
  verifying the claim.
\end{proof}

\begin{proof}[Proof of \eqref{eq:const-coeff-w-w*-approximation}]
  Let us assume that $B=3$. Then
  \begin{align} \label{eq:w-w*-1}
    \nabla_x^{l_x}\nabla_y^{\smash{l_y}} \phi_t^*(x,y)-D_x^{l_x}D_y^{\smash{l_y}} \phi_t(x,y)
    &= t^2 \int_{[-\pi,\pi]^d} W_t^*(a^*(\xi)) \hat\nabla^{l_x} \smash{\overline{\hat \nabla}}^{l_y} e^{i\xi\cdot (x-y)}\; d\xi
    \\
    &\quad - t^2 \int_{\R^d} W_t(a(\xi)) \hat D^{l_x} \smash{\overline{\hat D}}^{l_y} e^{i\xi\cdot (x-y)}\; d\xi
    .
    \nonumber
  \end{align}
  To simplify notation, we will write
  $\hat D^l = \hat D^{l_x} \smash{\overline{\hat D}}^{l_y}=\hat D^{l_x} \otimes \smash{\overline{\hat D}}^{l_y}$
  if $l=(l_x,l_y)$, and similarly for $\nabla$.
  Then the difference \eqref{eq:w-w*-1}
  may be estimated as follows. Proposition~\ref{prop:W-approximation} implies
  \begin{multline}
    \int_{[-\pi,\pi]^d} |W_t^*(a^*(\xi)+m^2) - W_t(a^*(\xi)+m^2)| |\hat D^l(\xi)| \; d\xi
    \\
    \leq 
    C t^{-1} \int_{\R^d} (1+C|\xi|^2 t^2 + m^2 t^2)^{-p-k} |\xi|^{l} \; d\xi
    \leq C t^{-d-l-1} (1+m^2t^2)^{-k}
  \end{multline}
  where we have assumed in the second inequality above that $p$ was chosen sufficiently large
  so that the integral is convergent.
  Similarly,
  we may proceed for the other differences, always choosing $p$ large enough in the estimates.
  Using \eqref{eq:W-decay-derivatives} with $m=1$
  and $|a^*(\xi) - a(\xi)| = O(|\xi|^3)$, which follows from Taylor's theorem, we obtain
  \begin{multline}
    \int_{[-\pi,\pi]^d} |W_t(a^*(\xi)+m^2) - W_t(a(\xi)+m^2)| |\hat D^l(\xi)| \; d\xi
    \\
    \leq 
    C \int_{\R^d}|\xi| (1+C|\xi|^2 t^2 + m^2 t^2)^{-p-k} |\xi|^{l} \; d\xi
    \leq
    C t^{-d-l-1} (1+m^2t^2)^{-k}
    .
  \end{multline}
  Taylor's theorem similarly implies that $|\hat \nabla^l(\xi) - \hat D^l(\xi)| \leq C |\xi|^{l+1}$ so that,
  by \eqref{eq:W-decay},
  \begin{multline}
    \int_{[-\pi,\pi]^d} |W_t^*(a^*(\xi)+m^2)| |\hat \nabla^l(\xi) - \hat D^l(\xi)| \; d\xi
    \\
    \leq 
    C \int_{\R^d} (1+C|\xi|^2t^2+m^2t)^{-p-k} |\xi|^{l+1} \; d\xi
    \leq
    C t^{-d-l-1} (1+m^2t^2)^{-k}
    .
  \end{multline}
  Finally, we obtain by \eqref{eq:W-decay} that
  \begin{multline}
    \int_{\R^d\setminus [-\pi,\pi]^d} |W_t(a(\xi)+m^2)| |\hat D^{l}(\xi)| \; d\xi
    \\
    \leq 
    C \int_{\R^d\setminus [-\pi,\pi]^d} (1+C|\xi|^2t^2+m^2t^2)^{-p-k} |\xi|^{l}\; d\xi
    \leq
    C t^{-2p} (1+m^2t^2)^{-k}
    .
  \end{multline}
  The combination of the previous four inequalities gives \eqref{eq:const-coeff-w-w*-approximation}.
\end{proof}

\section*{Acknowledgements}

The author thanks David Brydges and Gordon Slade for many helpful discussions,
advice, and careful proofreading. He also thanks Martin Barlow for helpful discussions.

\bibliographystyle{plain}

\def\polhk#1{\setbox0=\hbox{#1}{\ooalign{\hidewidth
  \lower1.5ex\hbox{`}\hidewidth\crcr\unhbox0}}}\def\polhk#1{\setbox0=\hbox{#1}{\ooalign{\hidewidth
  \lower1.5ex\hbox{`}\hidewidth\crcr\unhbox0}}}\def\polhk#1{\setbox0=\hbox{#1}{\ooalign{\hidewidth
  \lower1.5ex\hbox{`}\hidewidth\crcr\unhbox0}}}\def\cprime{$'$}

\end{document}